\documentclass[11pt]{article}
\usepackage{amsmath}
\usepackage{amssymb,amsbsy,amsthm}
\usepackage{graphicx}
\usepackage[dvipsnames]{xcolor}
\usepackage{mathtools}
\usepackage{enumerate}
\usepackage{cases}
\usepackage[margin=1in]{geometry}
\usepackage{pdfsync}
\usepackage{hyperref}
\usepackage{wrapfig} 
\usepackage{bbm}

\allowdisplaybreaks[1]

\numberwithin{equation}{section}

\newcommand*{\email}[1]{\href{mailto:#1}{#1}}

\newcommand{\p}{\mathbb{P}} 

\newtheorem{theorem}{Theorem}[section]
\newtheorem{lemma}[theorem]{Lemma}
\newtheorem{proposition}[theorem]{Proposition}
\newtheorem{claim}[theorem]{Claim}

\newtheorem{definition}[theorem]{Definition}

\newtheorem{remark}[theorem]{Remark}

\begin{document}

\title{Law of large numbers for activated random walk on villages}
    
\author{
    Balázs~Ráth \footnote{Department of Stochastics, Institute of Mathematics, Budapest University of Technology and Economics,
M\H{u}egyetem rkp.\ 3., H-1111 Budapest, Hungary} \footnote{HUN-REN Alfr\'ed R\'enyi Institute of Mathematics, Re\'altanoda utca 13-15, H-1053 Budapest,
Hungary } 
\email{rathb@math.bme.hu} \and
    Jacob~Richey\footnotemark[2]. \email{jrichey@renyi.hu}\and
    Miklós~Salánki\footnotemark[1]. \email{salanki.miklos@edu.bme.hu}
    }

\maketitle

\begin{abstract} 
We consider activated random walk (ARW), an interacting particle system and prototypical model of self-organized criticality in a setting which combines mean-field behavior with the geometry of an arbitrary graph, which we call the {\em{village model of ARW}}, or VARW for short. VARW is obtained from a fixed graph by replacing each vertex with a 'village' that consists of $n$ replicas of that vertex. We focus on VARW where particles walk according to a strictly sub-stochastic transition kernel on a finite underlying graph, so mass is sometimes lost (which guarantees that the system eventually stabilizes almost surely). Under a subcriticality assumption on the initial state we prove a law of large numbers as  $n$ goes to infinity for the resulting stable configuration of particles and the odometer of the process, to a limit which is uniquely characterized by a system of non-linear equations.
\end{abstract}

\section{Introduction} 
	\subsection{Background}
	Motivated by \emph{self-organized criticality} (SOC), a statistical property observed in real-world systems including forest fires, earthquakes, and electrical networks that describes the size of cascading avalanches by a power law, physicists Bak, Tang and Weisenfeld introduced a number of interacting particle models in the 1980s. Dhar carried out the first in-depth study of one such model, named the abelian sandpile model for a commutative property it exhibits. {\emph{Activated random walk}} (ARW) is a stochastic particle model which shares many features of the abelian sandpile -- most importantly, its abelian property -- and which conjecturally (and experimentally) exhibits the key features of SOC. ARW has a number of technical advantages that make it tractable, but many conjectured SOC properties are still unproved for ARW. Two open problems of particular interest are the existence of a (spherical) limit shape for ARW on $\mathbb{Z}^d$ \cite[Conjecture 1]{levineSilvestri2024}, and universality of the critical density \cite[Conjecture 15]{levineSilvestri2024}. 

    The focus of this work is a variant of ARW we call the village model of activated random walk (VARW), which combines properties of ARW on the complete graph with ARW on an arbitrary graph. Our main result is a law of large numbers (LLN) which describes the final stable configuration of sub-critical ARW in terms of a system of equations. In forthcoming work with Antal A.\ J\'arai we plan to prove a shape theorem for VARW in $\mathbb{Z}^d$, and to study the universality of critical density in the context of the stabilization/metastability dichotomy in the fixed energy version of VARW.

    Contrary to most existing work on ARW, our proof does not rely heavily on the abelian property (see e.g. \cite[Section 2]{rolla2020}) to construct specialized toppling sequences. Rather, we prove that the total odometer function, which counts the amount of activity at each vertex, is the fixed point of a map (which we call the single-loop odometer function)
    which has a certain contractive property. This observation greatly simplifies the proof of our main theorem, and we believe it may be useful for studying ARW more generally. 

    The VARW has many features in common with the divisible sandpile model, for which many of the conjectured properties of ARW and sandpile models are proved rigorously, including the existence of a spherical limit shape and the behavior of the model at the critical density \cite{bourabee2026, levineperes2009, levinemuruganperesugurcan2016}. Detailed results related to the critical density are also known for ARW on a large complete graph \cite{dickman2010, jarai2023, junge2026scaling} and for $\mathbb{Z}^d$ when $d \to \infty$ \cite{junge2025mf}. Similar settings to the VARW for general particle models have appeared in the recent works \cite{kaiser2026, faul2025}.

\section{Preliminaries} Let $\mathbb{Z}$ denote the set of integers.
Let $\mathbb{N}_+:=\{1,2,\dots\}$, $\mathbb{N}_0:=\{0,1,2,\dots \} $.

Let $\mathbb{R}_+ :=[0,+\infty)$ and $[n]:=\{1,2, \dots, n\}$. Given a finite set $V$ and a vector $w \in \mathbb{R}^V$,  we define its sup-norm by $\Vert w\Vert_{\infty}=\max_{x \in V} |w_x| $ and its $\ell^1$-norm by $\Vert w\Vert_{1}=\sum_{x \in V} |w_x| $.

\subsection{The continuous-time VARW}\label{ss_vm_jk_sr}

   Fix a finite set $V$ (the set of villages). Let  $P=(P_{x,y})_{x,y \in V}$ denote a strictly sub-stochastic and irreducible transition kernel: that is, we assume  $P_{x,y} \geq 0$ for all $x,y \in V$,  $\sum_{y \in V}P_{x,y} \leq 1$ for all $x \in V$, $\sum_{y \in V}P_{x,y} < 1$ for at least one $x \in V$, and for any $x,y \in V$ there exists $\ell \in \mathbb{N}_0$ such that $P^\ell_{x,y} >0$. We also fix a vector $\lambda = \left( \lambda_x \right)_{x \in V} \in \mathbb{R}_+^{V}$, where $\lambda_x$ denotes the sleep rate in village $x$, an input parameter for the VARW dynamics.

Given $V$, $P$, $\lambda$ and a positive integer $n$, the set of sites is the Cartesian product
\begin{equation} V_n = V \times [n] = \{ \, (x, i): x \in V, \; i \in [n] \,\}. \end{equation}
We call $(x,i)$ the $i$'th house in village $x$. Each village has $n$ houses. 

   We now define the rules of the continuous-time ARW particle process on $V_n$ (this is simply the adaptation of the usual rules of ARW to our setting).
    The state at time $t$ consists of some number of particles in each house. Each particle is either active or sleepy. All 
      particles in houses with two or more particles are always active, while particles that are alone in a house may be sleepy or active. Each active particle jumps at rate 1. A particle that jumps from a house in  village $x$ 
      either disappears  with probability $1-\sum_{y \in V} P_{x,y}$ or
      chooses village $y$ with probability $P_{x,y}$ and then picks a house in village $y$ uniformly  and  lands there. In other words, an active particle in house $(x,i)$ jumps out of the system at rate $1-\sum_{y \in V} P_{x,y}$ and jumps
      to house $(y,j)$ at rate $\frac{1}{n}P_{x,y}$. Each lonely particle in  village $x$ goes to sleep at rate $\lambda_x$. Sleeping particles do not move, but a sleeping particle immediately wakes up and becomes active  whenever another active particle jumps to the same house.

\subsection{Initialization and stabilization} \label{sec:arwstab}

We will run the VARW dynamics started from a specific (random) initial configuration of particles. Fix $\sigma_x \in [0,1]$ and $\nu_x \in \mathbb{R}_+$ for each $x \in V$, and let 
\begin{equation}\label{eq_sigma_nu_initial}
\sigma=(\sigma_x)_{x \in V}, \qquad \nu=(\nu_x)_{x \in V}.    
\end{equation}

The initial configuration consists of both sleepy and active particles: for each $x \in V$ we choose
$\lfloor \sigma_x n \rfloor$ houses in village $x$ (the exact location of these houses will not matter) and put one sleepy particle in each of these houses; and we add $\lfloor \nu_x n \rfloor$
many active particles to village $x$, choosing a house for each of them independently and uniformly.
 
Recall that the jump kernel $P$ is assumed to be irreducible and strictly sub-stochastic, so that a particle that jumps around too much will eventually disappear from the system. This guarantees that the VARW particle system almost surely terminates in a final configuration where each house contains at most one sleeping particle. We will refer to this process as {\em{stabilization}}, and call the resulting particle configuration (consisting of only sleeping particles) {\em{stable}}. 
Our main theorem will describe this final stable configuration of particles when $n$ is very large. To this end, for each $x \in V$ we define 
\begin{equation}\label{sleepy_S_n_def_eq}
S_*^{(n)}(x)=  S_*^{(n)}(\nu, \sigma; x) := \begin{array}{c} \text{ number of sleepy particles in village $x$} \\ \text{in the final stable configuration}\end{array} \end{equation}
and the jump odometers
\begin{equation}\label{def_eq_M_star_jump_x} 
M_*^{(n)}(x)=
M_*^{(n)}(\nu, \sigma; x) :=\begin{array}{c}
   \text{   number of times a particle jumped from  } \\
     \text{ a house in village  $x$ during stabilization.}
\end{array} \end{equation}
Denote these  $\mathbb{N}_0^V$-valued random variables by 
 \begin{equation}\label{def_eq_M_star_jump}
M_*^{(n)}=\left(  M_*^{(n)}(x) \right)_{x \in V} \quad \text{and} \quad S_*^{(n)}=\left(  S_*^{(n)}(x) \right)_{x \in V}.
\end{equation}

\section{Statement of main result: Law of Large Numbers} 

Our main result proves that the scaled quantities
\begin{equation}\label{eq_mn_sn}
m_*^{(n)} := \frac{1}{n}  M_*^{(n)} \qquad \text{ and } \qquad s_*^{(n)} := \frac{1}{n}  S_*^{(n)}  \end{equation} 
converge to deterministic limits $s_* \in [0,1]^V$ and $m_* \in \mathbb{R}_+^V$ as $n \to \infty$. We also formulate a system of equations for the limit vectors $s_*$ and $m_*$, and prove that under a sub-criticality type assumption on the initial configuration, this system of equations uniquely determines the solution $(s_*, m_*) \in [0,1]^V \times \mathbb{R}_+^V$. 

\begin{theorem}[Law of large numbers] \label{thm:LLN} Given $V$,  $P$, $\lambda$ as in Section \ref{ss_vm_jk_sr} and initial conditions
$\sigma$ and $\nu$ (cf.\ \eqref{eq_sigma_nu_initial}) that also satisfy $\sigma_x \leq \frac{\lambda_x}{1+\lambda_x}$ for all $x \in V$, the system of equations
\begin{eqnarray} 
\label{balance} 
s_x  & = & - m_x + \sigma_x +\nu_x + \sum_{y} m_y P_{y,x}, \qquad  x \in V \\ 
\label{lastExit} 
s_x & = & \sigma_x e^{-(\nu_x + \sum_y m_y P_{y,x})} + \frac{\lambda_x}{1+\lambda_x} \left(1-e^{-(\nu_x + \sum_y m_y P_{y,x})}\right),  \qquad  x \in V, \end{eqnarray}
has a unique solution $s_* \in [0,1]^V$ and $m_* \in \mathbb{R}_+^V$ and the random vectors $s_*^{(n)}$ and $m_*^{(n)}$ converge in probability to the deterministic limits $s_*$ and $m_*$.
\end{theorem}

\begin{remark}
The equations \eqref{balance}$\&$\eqref{lastExit} have the following simple heuristic meanings. Equation \eqref{balance} is a mass balance statement, capturing the fact that at each village the final mass is equal to the initial mass plus the influx minus the outflux. Equation \eqref{lastExit} is a local condition for the probability $s_x$ that there is a sleeping particle in house $i$ in village $x$ in the final configuration. Namely, assuming $\frac{1}{n} M_*^{(n)} \approx m$ for some deterministic $m \in \mathbb{R}_+^V$, the total number of particles that land in house $i$ of village $x$ in the entire process can be approximated by Poisson distribution with mean $\nu_x + \sum_y m_y P_{y,x}$ if $n \gg 1$. Then note that a given house has a sleeping particle in the terminal configuration if either the influx was zero and it had a sleeping particle to begin with, or if there was a positive influx and some particle was left sleeping there. Equation \eqref{lastExit} gives an approximate expression for the probability $s_x$ of this event when $n \gg 1$.
\end{remark}

\begin{remark} The solution $s_* =s_*(\sigma,\nu)$ and $m_*=m_*(\sigma,\nu)$ of \eqref{balance}$\&$\eqref{lastExit} also satisfies $0 \leq s_*(x) \leq \frac{\lambda_x}{1+\lambda_x}$ for all $x$, just like $\sigma$.
It is straightforward to check  that the identities \begin{equation*}s_*(s_*(\sigma,\nu),\nu')=s_*(\sigma,\nu+\nu'), \qquad
m_*(\sigma,\nu)+ m_*(s_*(\sigma,\nu),\nu')=m_*(\sigma,\nu+\nu')\end{equation*}
hold: these are the continuum analogs of the abelian property of ARW. Also note that  $\sigma^c:=\left( \frac{\lambda_x}{1+\lambda_x}\right)_{x \in V}$ satisfies $\sigma^c=s_*(\sigma^c,\nu)$ for any $\nu \in \mathbb{R}^V_+$, i.e. $\sigma^c$ is the continuum analog of the stationary stable configuration of the driven-dissipative ARW. If $|V|=1$, i.e., if there is only one village then this density agrees with the critical density $\frac{\lambda}{1+\lambda}$ for ARW on the complete graph with sleep rate $\lambda$, cf.\ \cite{jarai2023, junge2026scaling}. 
\end{remark}

\section{Proof of Theorem \ref{thm:LLN}}

The proof of the LLN will go along the following lines: we will first re-write the limiting equations \eqref{balance} and \eqref{lastExit} as the fixed point of a function with nice contractivity properties. Then we will show that the observables $s_*^{(n)}$ and $m_*^{(n)}$ (cf.\ \eqref{eq_mn_sn}) approximately solve this fixed point problem when $n \gg 1$. 

Throughout this section, we fix  $V, P, \lambda, \nu, \sigma$ and suppress the dependence of our various objects on these parameters.

\subsection{Fixed point formulation} 

Equating the RHS of \eqref{balance} with the RHS of \eqref{lastExit} and rearranging, we obtain a useful fixed point equation for the jump odometer.
Having fixed $\sigma, \nu$ and $P$, for any $x \in V$ let us define the affine linear maps  $\beta: \mathbb{R}^V_+ \to \mathbb{R}^V_+$ and $s: \mathbb{R}^V_+ \to \mathbb{R}^V$ by $\beta(m)=\nu+mP$ and $s(m)=-m+\sigma+\beta(m)$, i.e.,
\begin{equation}\label{def_eq_beta}
    \beta(m)_x:= \nu_x +\sum_{y} m_y P_{y,x}, \qquad s(m)_x := - m_x + \sigma_x +\nu_x + \sum_{y} m_y P_{y,x}.
\end{equation}
\begin{definition}[Fixed point function of VARW]
Define the function $\phi:\mathbb{R}_+^{V} \rightarrow{\mathbb{R}^{V}}$ as follows: for $m \in \mathbb{R}_+^{V}$ and $x \in V$, set
\begin{equation} \label{fixptfun}
     \phi(m)_x = (\sigma_x - \frac{\lambda_x}{1+\lambda_x})\left( 1-e^{-\beta(m)_x}\right) + \beta(m)_x
\end{equation}
We call $\phi$ the fixed point function of the VARW model.
\end{definition} 

If  $\phi(m_*) = m_*$ and if we define $s_*:= s(m_*)$ (cf.\ \eqref{def_eq_beta}) 
then $(m_*, s_*)$ solves
\eqref{balance}$\&$\eqref{lastExit}. Conversely, if $(m_*, s_*)$ solves \eqref{balance}$\&$\eqref{lastExit}, then $\phi(m_*) = m_*$ and  $s_*=s(m_*)$. 

 Crucially, in the sub-critical setting, $\phi$ has a unique such fixed point: 

\begin{proposition}[Existence and uniqueness of fixed point] \label{thm:fpsol} If $\sigma_x \leq \frac{\lambda_x}{1+\lambda_x}$ for all $x \in V$  then $\phi$ maps $\mathbb{R}_+^V$ to $\mathbb{R}_+^V$ and  $\phi$ has a unique fixed point in $\mathbb{R}_+^V$, denoted by $m_* = m_*( P, \lambda,\sigma,\nu)$.
\end{proposition}

Before we prove Proposition \ref{thm:fpsol}, we define a relevant norm on $\mathbb{R}^V$.

\begin{definition}[Principal eigenvector norm]\label{def_eta_norm}
     By the Perron-Frobenius theorem, the principal eigenvalue $\mu$ of the strictly sub-stochastic and irreducible matrix $P$ satisfies $\mu \in (0,1)$, the multiplicity of $\mu$ is equal to $1$ and there is a unique right eigenvector $\eta$ with non-negative coordinates that  satisfies $P\eta=\mu \eta$, moreover $\Vert \eta \Vert_{\infty}=1$ and
     \begin{equation}\label{eta_min_def}
         \eta_{min} := \min_{x \in V} \eta_x >0.
     \end{equation}
      We define the $\eta$-norm of a vector $w\in \mathbb{R}^{V}$ by 
      \begin{equation}\label{eta_norm_def_eq}
      \Vert w \Vert_{\eta} = \sum_x |w_x| \cdot  \eta_x.\end{equation}
\end{definition}
Note that for all $w \in \mathbb{R}^V$ we have
\begin{equation}\label{norm_inequalities}
 \eta_{min} \cdot \Vert w \Vert_1 \leq \Vert w \Vert_{\eta} \leq \Vert w \Vert_1, \qquad   \Vert w \Vert_\infty \leq \Vert w \Vert_1 \leq |V| \cdot \Vert w \Vert_\infty.
\end{equation}

\begin{proof}[Proof of Proposition \ref{thm:fpsol}] First observe that  $\sigma_x \leq \frac{\lambda_x}{1+\lambda_x}$, $\nu_x \geq 0$ and $m_x \geq 0$ for all $x \in V$ together imply that $\phi(m)_x \geq 0$ for all $x \in V$, and thus $\phi$ maps $\mathbb{R}_+^{V}$ to itself.
We will show that  $\phi$ is a contraction on $\mathbb{R}_+^{V}$ with respect to the metric induced by the $\Vert \cdot \Vert_\eta$-norm.

 We compute the directional derivative $\phi^{'}_{v}(m)$ of $\phi$ at $m \in \mathbb{R}_+^V$ in direction $v \in \mathbb{R}^{V}$ as
\begin{equation}  \label{dirder}
    \phi^{'}_{v}(m)_x  =\lim_{h \xrightarrow{}0} \frac{\phi(m+hv)_x-\phi(m)_x}{h} 
     =  \left( 1 +(\sigma_x - \frac{\lambda_x}{1+\lambda_x})\cdot e^{-\beta(m)_x} \right) \cdot \sum_{y} v_y P_{y,x}
\end{equation}
Since $\nu$, $m$ and $P$ are non-negative, and $0 \leq \sigma_x \leq  \frac{\lambda_x}{1+\lambda_x}$ for all $x \in V$, we see that 
\begin{equation}\label{compl_abs_leq_1}
\left|1 +\left(\sigma_x - \frac{\lambda_x}{1+\lambda_x}\right)\cdot e^{-\beta(m)_x}\right| \leq 1, \qquad x \in V.\end{equation}
Combining the above facts, we upper bound the $\eta$-norm of $\phi^{'}_{v}(m)$ by 
\begin{equation}\label{dir_der_short}
    \Vert \phi^{'}_{v}(m) \Vert_{\eta}  \stackrel{\eqref{eta_norm_def_eq} }{=}
   \sum_{x}  |\phi^{'}_{v}(m)_x| \cdot \eta_x \stackrel{\eqref{dirder}, \eqref{compl_abs_leq_1} }{\leq} 
      \sum_{x,y} |v_y| P_{y,x} \eta_x = \sum_y |v_y| \mu \eta_y \stackrel{\eqref{eta_norm_def_eq}}{=}
      \mu \Vert v \Vert_{\eta}.
\end{equation}
For any $m', m'' \in \mathbb{R}_+^V$, integrating \eqref{dir_der_short} along the line segment that joins $m'$ and $m''$ we obtain 
$\Vert \phi(m')-\phi(m'')\Vert_\eta \leq \mu \Vert m'-m'' \Vert_\eta $. Noting that $\mu \in (0,1)$, the existence and uniqueness of the fixed point $m_*$ of $\phi$ then follow from the Banach fixed point theorem. 
\end{proof} 

\begin{lemma} \label{lem:fpbound} Under the assumption of Proposition \ref{thm:fpsol}, for any $m \in \mathbb{R}_+^V$ we have
\begin{equation}\label{fixed_point_control}
\Vert m - m_*\Vert_\eta \leq \frac{1}{1-\mu} \Vert m - \phi(m)\Vert_\eta, \end{equation}
where $\mu < 1$ is the principal eigenvalue of $P$ (cf.\ Definition \ref{def_eta_norm}). \end{lemma}

\begin{proof} This is a standard 'residual' bound using the fact that $\phi$ is a contraction with contraction factor $\mu < 1$ for the metric induced by the norm $\Vert \cdot \Vert_\eta$, cf.\ \cite[Section 3.2]{gd2003}.
\end{proof}

\subsection{Discrete version of fixed point function}

We now relate the VARW to the fixed point equation $\phi(m_*) = m_*$  by constructing a discrete variant $\phi^{(n)}$ of the map $\phi$ associated to the VARW process (for fixed, finite $n$) so that $\phi^{(n)}( m_*^{(n)})= m_*^{(n)}$ holds (where $m_*^{(n)}$ was defined in \eqref{eq_mn_sn}). 

We start by building a probability space where VARW -- in particular, its terminal stable configuration -- can be defined. 

\begin{definition}[Stacks of instructions] We define three independent families of random variables, called {\emph{instructions}}, as follows: 

\begin{itemize} \item[(Airplane tickets)] For $x \in V$, let $\zeta_{j,x}, \, j \in \mathbb{N}_+ $ be i.i.d. random variables sampled according to the transition kernel $P(x, \cdot)$, i.e. $\p(\zeta_{j,x} = y) = P_{x,y}$ for $y \in V$, and $\p(\zeta_{j,x} = \dagger) = 1-\sum_{y \in V} P_{x,y}$, where $\dagger$ is a designated graveyard state. 

\item[(Taxi tickets)] For $x \in V$, let $\gamma_{j,x},\, j \in \mathbb{N}_+$ be i.i.d. uniform random variables on $[n]$. 

\item[(Landlord notices)] For $x \in V$ and $i \in [n]$, let $\kappa_{j,(x,i)}, j \in \mathbb{N}_+$ be i.i.d.\ Bernoulli random variables, taking value $0$ (sleep) with probability $\frac{\lambda_x}{1+\lambda_x}$ and value $1$ (jump) with probability $\frac{1}{1+\lambda_x}$. 
\end{itemize}
\end{definition}

Now we describe how these ingredients (together with some auxiliary exponential clocks) can be used to construct the jump odometers $ M_*^{(n)}$, cf.\ \eqref{def_eq_M_star_jump}.

\begin{definition}[Construction of VARW using stacks] \label{def_arw_using_stacks}
For each $x \in V$ we start with
  $\lfloor \sigma_x n \rfloor$ houses in village $x$ with one sleepy particle in them. 
  In each operation below we always execute the instructions of each stack in ascending order: $j=1,2,\dots$

\begin{enumerate}

\item \label{stacks_def_i} For each $x \in V$, $\lfloor \nu_x n \rfloor$ many immigrants arrive at the airport of village $x$ and use the first $\lfloor \nu_x n \rfloor$ taxi tickets  $\gamma_{j,x}, \, j=1,\dots, \lfloor \nu_x n \rfloor$ to distribute themselves in the houses of village $x$.

\item \label{stacks_def_ii}
Each active particle in village $x$ has a clock with rate $1+\lambda_x$. When the first clock rings (in any of the houses of any of the villages), say in house $i$ of village $x$, we reveal the next unused landlord notice $\kappa_{j, (x,i)}$ of house $(x,i)$ and evolve the system by making that active particle sleep or jump accordingly. Note that if the instruction is a sleep and the house $(x,i)$ has at least one other (active) particle in it, then nothing happens (the particle tries to sleep and is then instantly re-awoken). 

\item \label{stacks_def_iii} If the revealed landlord notice $\kappa$  was a jump instruction, then this particle goes to the airport of its current village $x$, grabs the next airplane ticket $\zeta_{j',x}$ and
flies to a village (or the graveyard state $\dagger$) accordingly.
If the airplane arrived at village $y$
then it moves to a house in village $y$ according to the next taxi ticket $\gamma_{j'', y}$.

\item We repeat steps \ref{stacks_def_ii} and \ref{stacks_def_iii} until all particles have either jumped to the graveyard state or went to sleep.
 \end{enumerate} 

\end{definition} 

The following well-known fact shows that the final configuration of sleeping particles does not depend on the values of the exponential clocks: 

\begin{lemma}[Abelian property] \label{lem:abelian} The stabilizing odometer configuration  $M_*^{(n)}$ almost surely does not depend on the order of active particles toppled. \end{lemma} 

\begin{proof} See for example Section 2 of \cite{rolla2020}. The strict sub-stochasticity and irreducibility of $P$ guarantees that the process eventually terminates almost surely. \end{proof}

We now build a discrete variant $\phi^{(n)}: \mathbb{R}_+^{V} \to \mathbb{R}_+^V$ of the map $\phi$ using
the  instructions $\zeta, \gamma, \kappa$. We first define the single-loop odometer function $\Phi^{(n)}: \mathbb{N}_0^{V} \to \mathbb{N}_0^{V}$ and later we will define its scaled variant 
$\phi^{(n)}(m):=\frac{1}{n}\Phi^{(n)}(\lfloor n m \rfloor)$. The idea of the construction is: given any (hypothetical) input jump odometer configuration $M \in \mathbb{N}_0^V$ we use the $\zeta, \gamma, \kappa$ stacks 
to trace the movement of  particles until their next jump, and produce an output jump odometer configuration $\Phi^{(n)}(M) \in \mathbb{N}_0^V$ as a result.

\begin{definition}[Single-loop jump odometer function]\label{def_single_loop}
Given $n \in \mathbb{N}_+$ and an input jump odometer configuration $M = \left(M_x\right)_{x \in V} \in \mathbb{N}_0^{V}$ for $x \in V$, we define  $\Phi^{(n)}(M)$ as follows:

\begin{itemize} 

\item[(Inbound)] Calculate the number of particles that arrived to the airport of village $x \in V$:

\begin{equation}\label{def_eq_I_x}
I_x = \lfloor  \nu_x n  \rfloor + \sum_{y \in V} \sum_{j=1}^{M_y} \mathbb{1}\{\zeta_{j,y} = x\} 
\end{equation}

\item[(Active houses)] Distribute the particles arriving to the airport of village $x$ among the houses in village $x$ using the taxi tickets. Denote by $\mathcal{A}_x$ the set of houses in village $x$ in which at least one active particle arrived, and let $A_x$ denote the number of such houses:
\begin{equation}\label{def_eq_active_houses_in_x}
   \mathcal{A}_x:= \left\{ i \in [n] \,: \; \exists \, j \in [I_x] \,: \, \gamma_{j, x} = i \, \right\}, \qquad A_x:= |\mathcal{A}_x|
\end{equation}
\item[(Quiet houses)] Let $\sigma_{x,i}=1$ if the house $(x,i)$ initially had a sleepy particle and $\sigma_{x,i}=0$ otherwise.
Denote by $Q_x$ the number of houses in village $x$ 
in which there was an initial sleepy particle, but no active particle arrived there:
\begin{equation}\label{def_eq_quiet_houses_in_x}
   Q_x:= \sum_{i=1}^n \mathbb{1}\{ \,  \sigma_{x,i}=1, \; \forall \, j \in [I_x] \,: \, \gamma_{j, x} \neq i \, \} 
\end{equation}

\item[(Jump/Sleep)] If $i \in \mathcal{A}_x$, i.e., if $(x,i)$ is an active house of village $x$, let us define the $\mathbb{N}_+$-valued random variable
\begin{equation} T_{x,i}=\sigma_{x,i}+\sum_{j=1}^{I_x} \mathbb{1}\{\gamma_{j,x} =i \}. 
\end{equation}
Thus $ T_{x,i}$ is total number of active particles that ever entered this house (plus one in case they woke up a sleepy particle there).  Let
\begin{equation}\label{def_eq_N_xi}
N_{x,i}=\min\left\{ \, K \in \mathbb{N}_0 \, : \, \sum_{j=1}^K \kappa_{j,(x,i)} = T_{x,i}-1 
\right\}
\end{equation} 
denote the number of landlord notices revealed until $T_{x,i}-1$ particles jumped away from house $(x,i)$.  Thus, if $ \kappa_{N_{x,i}+1,(x,i)}=1$ then the last active particle in house $(x,i)$ jumps, but if  $ \kappa_{N_{x,i}+1,(x,i)}=0$ then $(x,i)$ will have a sleepy particle
in the final stable configuration. Let 
\begin{equation}\label{def_eq_J_x}
J_x=\sum_{i \in \mathcal{A}_x} \kappa_{N_{x,i}+1,(x,i)}
\end{equation} denote the number of active houses in village $x$ from which all particles jumped away (these houses will be empty in the final stable configuration).

\item[(Outbound)] 
Now we are ready to define $\Phi^{(n)}(M) \in \mathbb{N}_0^V$ by writing a formula for the number $\Phi^{(n)}(M)_x$ of particles that jumped from village $x$: we just have to add the number $I_x$ of incomers to the number $\lfloor \sigma_x n  \rfloor$ of particles who were initially there and subtract from it the number $Q_x+A_x-J_x$ of particles who stayed:
\begin{equation} \label{dfp} \Phi^{(n)}(M)_x = \lfloor \sigma_x n \rfloor -Q_x +I_x - A_x + J_x. 
\end{equation}
\item[(Net income)] Let us define the $\mathbb{Z}^V$-valued random variable $S^{(n)}(M)$ by 
  \begin{equation}\label{def_eq_S_n_function}
      S^{(n)}(M)_x := - M_x + \lfloor \sigma_x n \rfloor + I_x, \qquad x \in V. 
  \end{equation}
\end{itemize} 
For any $m=(m_x)_{x \in V} \in \mathbb{R}_+^V$, let 
$\lfloor m \rfloor := (\lfloor m_x \rfloor)_{x \in V} \in \mathbb{N}_0^V$.
Let us define the scaled random functions
  $\phi^{(n)}: \mathbb{R}_+^{V} \to \mathbb{R}_+^V$ 
  and  $s^{(n)}: \mathbb{R}_+^{V} \to \mathbb{R}^V$
  by
\begin{equation}\label{def_eq_scaled_phi} 
\phi^{(n)}(m):=\frac{1}{n}\Phi^{(n)}(\lfloor n m \rfloor), 
\qquad
s^{(n)}(m):=\frac{1}{n}S^{(n)}(\lfloor n m \rfloor). 
\end{equation}

\end{definition}

\begin{lemma}[Single-loop jump odometer fixed point] \label{lem:loop_abelian}  For almost all realizations of the instruction stacks $\zeta, \gamma, \kappa$,  the scaled jump odometer configuration $m_*^{(n)}$  and the scaled final sleepy particle counter $s_*^{(n)}$ (cf.\ \eqref{def_eq_M_star_jump}) satisfy
\begin{equation}\label{discrete_fixed_point_eq}
\phi^{(n)}( m_*^{(n)}) = m_*^{(n)}, \qquad s^{(n)}(m_*^{(n)})=s_*^{(n)}.
\end{equation}
\end{lemma}
\begin{proof} The identity $\Phi^{(n)}( M_*^{(n)}) = M_*^{(n)}$ follows from \eqref{def_eq_M_star_jump_x}, \eqref{def_eq_M_star_jump} and Definitions \ref{def_arw_using_stacks}, \ref{def_single_loop}. For any $M$, the number of sleepy particles in village $x$ after a single loop is equal to $Q_x+A_x-J_x$, thus
  \eqref{sleepy_S_n_def_eq}, 
\eqref{def_eq_M_star_jump},  $\Phi^{(n)}( M_*^{(n)}) = M_*^{(n)}$,
\eqref{dfp} and \eqref{def_eq_S_n_function} imply that
$S^{(n)}(M_*^{(n)})=S_*^{(n)}$ holds. Now \eqref{discrete_fixed_point_eq} follows from \eqref{eq_mn_sn} and \eqref{def_eq_scaled_phi}.
\end{proof}

\subsection{Proof of LLN using that $\phi^{(n)}$ concentrates around $\phi$}

Recall the definitions of $\phi(\cdot)$ and $\phi^{(n)}(\cdot)$  from \eqref{fixptfun} and \eqref{def_eq_scaled_phi}, respectively. 
Our main tool is that $\phi^{(n)}(m)$ and $\phi(m)$ are close if $n \gg 1$. We start by stating these concentration estimates and then show how they imply the main theorem. 
Recall the definitions of $s(\cdot)$ and $s^{(n)}(\cdot)$ from \eqref{def_eq_beta} and \eqref{def_eq_scaled_phi}, respectively.  For any $r \in \mathbb{R}$ let $r_+ := \max\{r,0\}$.

\begin{lemma}[Single-loop concentration bounds] \label{lem:mainbound} For all $a \in \mathbb{R}_+$ and $M \in \mathbb{N}_0^{V}$ we have
\begin{eqnarray}\label{eq_statement_single_loop_sleepy_concentr}
\p\left( \left\Vert s^{(n)}\left( \frac{M}{n} \right) - s\left( \frac{M}{n} \right) \right\Vert_\infty \geq a \right) \leq 2|V| \exp\left( -\frac{2 (an-2)_+^2 }{  \Vert M \Vert_1} \right)
\\
\label{eq_concentration_bound}
\p\left( \left\Vert\phi^{(n)}\left( \frac{M}{n} \right) - \phi\left( \frac{M}{n} \right) \right\Vert_\infty \geq a \right) \leq 
4 |V| \exp\left(-\frac{ 2  (an- \Vert \nu \Vert_1 - \Vert M \Vert_1/n -2 )_+^2 }
{ 81   \left( n+ \Vert \nu \Vert_1 n + an+  \Vert M \Vert_1 \right) } \right) 
\end{eqnarray}
\end{lemma}
Before we prove Lemma \ref{lem:mainbound} in Section \ref{subsection_proof_lemma_concentration}, we derive the proof of Theorem \ref{thm:LLN} from it. We begin with the following crude estimate. Recall the definition of $M_*^{(n)}$ from
\eqref{def_eq_M_star_jump}.

\begin{lemma} \label{lem:odotail} There exists $C_*=C_*(P,\nu) \in \mathbb{R}_+$  such that
\begin{equation}\label{p_m_star_big_goestozero_statement_eq}
\p(\Vert M_*^{(n)}\Vert_1 > C_* n) \to 0 \quad \text{ as } \quad n \to \infty. \end{equation} 
\end{lemma} 
\begin{proof}
First observe that \eqref{eta_norm_def_eq} and $P \eta =\mu \eta$ imply that $\Vert m P\Vert_\eta = \mu \cdot \Vert m \Vert_\eta$ where $\mu \in (0,1)$ is the principal eigenvalue of $P$. Moreover we can use $\phi(m)_x \geq 0$ and $0 \leq \sigma_x \leq \frac{\lambda_x}{1+\lambda_x}$ to bound $\phi(m)_x \leq \beta(m)_x$ (cf.\ \eqref{fixptfun}), obtaining 
\begin{equation}\label{eta_norm_bound_crude}
\Vert  \phi( m)\Vert_{\eta} \leq \Vert \beta(m)  \Vert_\eta \stackrel{\eqref{def_eq_beta}}{=} \Vert \nu + m P \Vert_\eta  \leq
\Vert \nu \Vert_\eta + \mu \cdot \Vert m \Vert_\eta. 
\end{equation}
Let $\alpha:= \frac{1-\mu}{2} >0$. On the event $A_n(m):=\{\,\Vert\phi^{(n)}(m) -  \phi( m)\Vert_{\eta} \leq \alpha \Vert m \Vert_\eta\,\}$ we have
\begin{equation}\label{strict_noneq}
 \Vert\phi^{(n)}(m)\Vert_{\eta} \leq \Vert\phi(m)\Vert_{\eta} + \alpha\Vert m\Vert_{\eta}  \stackrel{\eqref{eta_norm_bound_crude}}{\leq}  \Vert \nu \Vert_{\eta} + \frac{1+\mu}{2} \cdot\Vert m \Vert_{\eta} \stackrel{(*)}{<} \Vert m \Vert_{\eta},
\end{equation}
where $(*)$ holds if we further assume $ \Vert m \Vert_{\eta} > \frac{1}{\alpha} \Vert \nu \Vert_{\eta} $. Recalling the definition of $\eta_{min}$ from \eqref{eta_min_def}, we define $C_*:= \frac{1}{ \eta_{min}} \frac{1}{\alpha} \Vert \nu \Vert_{\eta}$. It follows from \eqref{norm_inequalities} and \eqref{strict_noneq}  that if $\Vert m \Vert_1 > C_*$ and $A_n(m)$ holds then $\Vert\phi^{(n)}(m)\Vert_{\eta}<\Vert m \Vert_\eta$, and thus
$\phi^{(n)}(m) \neq m$. On the other hand, $\phi^{(n)}( m_*^{(n)}) = m_*^{(n)}$ holds by Lemma \ref{lem:loop_abelian},
and thus if $\Vert M_*^{(n)}\Vert_1 > C_* n$ then 
$A^c_n(m^{(n)}_*)$ must hold.
If we let $B_n(m):=\{ \Vert\phi^{(n)}(m) -  \phi( m)\Vert_\infty \geq \frac{1}{|V|} \eta_{min} \alpha \Vert m \Vert_1 \}$ then
by \eqref{norm_inequalities} we have $A_n^c(m) \subseteq B_n(m)$.
Therefore $\Vert M_*^{(n)}\Vert_1 > C_* n$ implies $B_n(M^{(n)}_*/n)$, and we obtain 
\begin{multline}\label{big_M_*_prob_bound_eq}
    \mathbb{P}(\, \Vert M_*^{(n)} \Vert_1 > C_*n \, ) \leq
\mathbb{P}\left( \, \exists \, M \,: \, \Vert M \Vert_1 > C_* n \; \text{ and } \; B_n(M/n) \text{ holds} \,   \right) \leq \\
      \sum_{k \geq C_* n} \; \sum_{M \, : \, \Vert M \Vert_1=k } \mathbb{P}\left(\,  B_n\left(M/n\right)\,  \right)    \stackrel{\eqref{eq_concentration_bound}}{\leq}  \\   \sum_{k \geq C_* n} (k+1)^{|V|} \cdot 4\exp\left(-\frac{ 2  (  \eta_{min} \alpha k / |V| - \Vert \nu \Vert_1 - k/n -2 )_+^2 }
{ 81   \left( n+ n\Vert \nu \Vert_1  + \eta_{min} \alpha k/|V| +  k \right) } \right).
\end{multline}
The sum on the RHS of \eqref{big_M_*_prob_bound_eq} goes to zero as $n \to \infty$, i.e., \eqref{p_m_star_big_goestozero_statement_eq} holds.
\end{proof}

\begin{proof}[Proof of Theorem \ref{thm:LLN}] By Lemma \ref{lem:odotail}, in order to prove that $m_*^{(n)}$ converges in probability to $m_*$, it is enough to show that for any $\varepsilon>0$ we have
\begin{equation}
 \mathbb{P}\left( \Vert m_*^{(n)} -m_* \Vert_\eta \geq \varepsilon, \, \; 
 \Vert  m_*^{(n)} \Vert_1 \leq C_*
 \right) \to 0, \qquad n \to \infty.
\end{equation}
We can bound 
\begin{multline}\label{bbbbound}
\Vert m_*^{(n)} -m_* \Vert_\eta \stackrel{\eqref{fixed_point_control}}{\leq} \frac{1}{1-\mu} \Vert m_*^{(n)} - \phi(m_*^{(n)}) \Vert_\eta
\stackrel{\eqref{discrete_fixed_point_eq}}{=} \\
\frac{1}{1-\mu} \Vert \phi^{(n)}(m_*^{(n)}) - \phi(m_*^{(n)}) \Vert_\eta \stackrel{\eqref{norm_inequalities}}{\leq} \frac{|V|}{1-\mu} \Vert \phi^{(n)}(m_*^{(n)}) - \phi(m_*^{(n)}) \Vert_\infty,
\end{multline}
and thus
\begin{multline}\label{m_star_bbound}
   \mathbb{P}\left( \Vert m_*^{(n)} -m_* \Vert_\eta \geq \varepsilon, \, \; 
 \Vert  m_*^{(n)} \Vert_1 \leq C_*
 \right) \stackrel{ \eqref{bbbbound} }{\leq} \\
 \mathbb{P}\left( \, \exists \, M \, : \, \Vert M \Vert_1 \leq C_* n \; \text{ and } \; \left\Vert\phi^{(n)}\left( \frac{M}{n} \right) - \phi\left( \frac{M}{n} \right) \right\Vert_\infty \geq \varepsilon \frac{1-\mu}{|V|}
  \right) \leq \\
  \sum_{ M \, : \, \Vert M \Vert_1 \leq C_* n }  
 \p\left( \left\Vert\phi^{(n)}\left( \frac{M}{n} \right) - \phi\left( \frac{M}{n} \right) \right\Vert_\infty \geq \varepsilon \frac{1-\mu}{|V|} \right) \stackrel{\eqref{eq_concentration_bound}}{\leq} \\
\left( C_* n \right)^{|V|} 
4 |V| \exp\left(-\frac{ 2  ( \varepsilon \frac{1-\mu}{|V|} n - \Vert \nu \Vert_1 - C_* -2 )_+^2 }
{ 81   \left( n+ \Vert \nu \Vert_1 n  +  \varepsilon \frac{1-\mu}{|V|} n +   C_* n \right) } \right) \to 0, \quad n \to \infty.
 \end{multline}
This completes the proof of $m_*^{(n)} \stackrel{\mathbb{P}}{\longrightarrow} m_*$. It remains to deduce $s_*^{(n)} \stackrel{\mathbb{P}}{\longrightarrow} s_*$ from this. Recalling
$s_*^{(n)}= s^{(n)}(m_*^{(n)})$ (cf.\ \eqref{discrete_fixed_point_eq}),  $s_*=s(m_*)$ (cf.\ \eqref{balance}, \eqref{def_eq_beta}), it is enough to show that $\left\Vert s^{(n)}\left( m^{(n)}_* \right) - s\left(   m^{(n)}_* \right) \right\Vert_\infty \stackrel{\mathbb{P}}{\longrightarrow} 0$. By Lemma \ref{lem:odotail}, it is enough to check that with
 high probability there is no $M$ with $\Vert M \Vert_1 \leq C_* n$ for which 
$\left\Vert s^{(n)}\left( M/n \right) - s\left( M/n \right) \right\Vert_\infty \geq \varepsilon$ holds. This follows from \eqref{eq_statement_single_loop_sleepy_concentr} by a calculation similar to \eqref{m_star_bbound}.
\end{proof}

\subsection{Proof of single-loop concentration bounds}\label{subsection_proof_lemma_concentration}

The goal of Section \ref{subsection_proof_lemma_concentration} is to prove Lemma \ref{lem:mainbound}.

Throughout we fix $x \in V$,  $a \in \mathbb{R}_+$ and $M \in \mathbb{N}_0^{V}$.
We use the shorthand 
\begin{equation} 
\Phi^{(n)}_x:= \Phi^{(n)}(M)_x, \qquad m:=\frac{M}{n}.
\end{equation}

The number $I_x$ of particles that arrived to village $x$ (cf.\ \eqref{def_eq_I_x}) is equal to the constant $\lfloor  \nu_x n  \rfloor$ plus  at most $\Vert M \Vert_1$ many independent Bernoulli variables, thus Hoeffding's inequality yields 
\begin{equation} \label{dev1}
\mathbb{P}(|I_x-\mathbb{E}(I_x)| \geq L) \leq 2\exp\left(-\frac{2L^2}{  \Vert M\Vert_1}\right)
\quad \text{for any} \quad  L \in \mathbb{R}_+.
\end{equation}

Recalling \eqref{def_eq_beta}, it follows from \eqref{def_eq_I_x} and the linearity of expectation that we have
\begin{equation}\label{EI_x_nbetax}
    \left| \mathbb{E}(I_x) -\beta(m)_x \cdot n \right| \leq 1.
\end{equation}
Now the proof of \eqref{eq_statement_single_loop_sleepy_concentr} directly follows
from \eqref{def_eq_S_n_function}, \eqref{def_eq_scaled_phi}, \eqref{def_eq_beta}, \eqref{dev1} and \eqref{EI_x_nbetax}.

It remains to prove \eqref{eq_concentration_bound}. Recall that $\Phi^{(n)}_x$ is defined by \eqref{dfp}. 
 The  distribution of $\Phi^{(n)}_x$ admits the following equivalent description.
\begin{claim}[Only the last landlord notice matters]\label{claim_alter_kappa} Fix a village $x \in V$. 
 Sample $n$ i.i.d.\ Bernoulli$(\frac{1}{1+\lambda_x })$ random variables $\widetilde{\kappa}_i, i \in [n]$ independently of all  stack instructions.
 Define the sum $\widetilde{\Phi}^{(n)}_x$ as in \eqref{dfp} but replace $J_x$
by $\widetilde{J}_x=\sum_{i \in \mathcal{A}_x} \widetilde{\kappa}_i$. Then $\Phi^{(n)}_x$ and $\widetilde{\Phi}^{(n)}_x$ have the same distribution.
\end{claim}

\begin{proof} The random variable $Q_x$, the random set $\mathcal{A}_x$ and the random variables $T_{x,i}, \, i \in [n]$ are all measurable with respect to the sigma-algebra $\mathcal{F}^{(n)}_{\zeta,\gamma}$ generated by  the airplane ticket and taxi ticket stacks.
Furthermore, the random variables
$N_{x,i}, \, i \in [n]$ (cf.\ \eqref{def_eq_N_xi}) have the following joint stopping time property:
for any collection $\left(n_{i}\right)_{i \in [n]} \in \mathbb{N}_0^{[n]}$ of non-negative integers, 
the event $\{ N_{x,i}=n_i, \, i \in [n] \}$ is determined by $\mathcal{F}^{(n)}_{\zeta,\gamma}$ and the landlord instructions 
$\left( \kappa_{j,(x,i)}, \, i \in [n], \, j=1,\dots, n_i \right)$. Thus, conditional on $\mathcal{F}^{(n)}_{\zeta,\gamma}$ and 
$\{ N_{x,i}=n_i, \, i \in [n] \}$, the random variables $\left( \kappa_{N_{x,i}+1,(x,i)}  \right)_{i \in \mathcal{A}_x}$ are i.i.d.\ 
with Bernoulli$(\frac{1}{1+\lambda_x })$ distribution. Comparing the definitions of $J_x$ (cf.\ \eqref{def_eq_J_x}) and $\widetilde{J}_x$, the statement of Claim \ref{claim_alter_kappa} follows.
\end{proof}

Next we bound the fluctuation of $\Phi^{(n)}_x$ if we condition on the event that the number $I_x$ of taxi tickets sold in village $x$ is equal to some  $u \in \mathbb{N}_0$. We will use the shorthand 
\begin{equation}\label{def_eq_E_u_P_u}
\mathbb{P}_u(\cdot):=\mathbb{P}( \, \cdot \, | \, I_x=u ), \qquad \mathbb{E}_u(\cdot):=\mathbb{E}( \, \cdot \, | \, I_x=u ).
\end{equation}

With Claim \ref{claim_alter_kappa} in hand, we are in a position to apply McDiarmid's inequality.  Note that for all $u$, there exists a function $f: [n]^u \times \{0,1 \}^n \to \mathbb{N}_0$ such that conditionally on $\{I_x = u\}$, the value of $\widetilde{\Phi}^{(n)}_x$ is obtained by applying $f$ to the $u$ taxi tickets sold (i.e., $\gamma_{1,x}, \dots, \gamma_{u,x}$)  and the random variables $\widetilde{\kappa}_i, \, i \in [n]$. Moreover, this $f$ has bounded differences: changing any particular taxi ticket can change the resulting outflux by at most $3$ (for example, one worst case is changing one taxi ticket from a house where the last landlord notice was a sleep to an unvisited house with an initial sleeping particle which then ejects two particles), while changing any of the $\widetilde{\kappa}_i$'s changes the outflux by at most $1$. Thus, by  McDiarmid's inequality, for all $u \in \mathbb{N}_0$ and $K \in \mathbb{R}$, we have
\begin{equation} \label{dev2}
\mathbb{P}_u\left(|\Phi^{(n)}_x-\mathbb{E}_u(\Phi^{(n)}_x )| \geq K  \right) \leq 2 \exp\left(-\frac{2}{9}\frac{K_+^2}{n+u}\right).
\end{equation}

\begin{claim}[Expected outflux given fixed influx] For any $x \in V$ and $u \in \mathbb{N}_0$ we have
\begin{equation}\label{Phi_cond_expect_on_u}
\mathbb{E}_u(\Phi^{(n)}_x) = n \cdot \left( \frac{\lfloor \sigma_x n \rfloor}{n} -\frac{\lambda_x}{1+\lambda_x}\right)\cdot \left(1-\left(1-\frac{1}{n}\right)^u \right)+u.
\end{equation}
\end{claim}
\begin{proof} Recall from \eqref{dfp} that we have $\Phi^{(n)}_x = \lfloor \sigma_x n \rfloor -Q_x +I_x - A_x + J_x$.  We have
$\mathbb{E}_u(Q_x)=\lfloor \sigma_x n \rfloor \cdot (1-\frac{1}{n})^u $
by \eqref{def_eq_quiet_houses_in_x}, $\mathbb{E}_u(A_x)=n \cdot (1-(1-\frac{1}{n})^u ) $ by \eqref{def_eq_active_houses_in_x}, and $\mathbb{E}_u(J_x)= n \cdot (1-(1-\frac{1}{n})^u )  \frac{1}{1+\lambda_x}$ by 
\eqref{def_eq_J_x} and Claim \ref{claim_alter_kappa}. Using the linearity of expectation to put these ingredients together, we obtain
\eqref{Phi_cond_expect_on_u}.
\end{proof}

Using  $\left| (1-\frac{1}{n})^u -e^{-u/n} \right| \leq \frac{u}{n^2}$ and $0 \leq \sigma_x \leq \frac{\lambda_x}{1+\lambda_x}$, we obtain 
\begin{equation}
\left| \left( \sigma_x   -\frac{\lambda_x}{1+\lambda_x}\right)\cdot \left(1-e^{-u/n} \right)+\frac{u}{n} -  \frac{1}{n}\mathbb{E}_u(\Phi^{(n)}_x) \right| \leq \frac{1}{n} +\frac{u}{n^2}.
\end{equation}
Thus, recalling \eqref{fixptfun} and \eqref{def_eq_beta},   using that $0 \leq \sigma_x \leq \frac{\lambda_x}{1+\lambda_x}$ we obtain that if $K \geq 0$ then
\begin{equation}\label{E_u_Phi_n_approx}
\left|u -n \beta( m )_x \right| \leq K \; \implies \;  \left|  \frac{1}{n}\mathbb{E}_u(\Phi^{(n)}_x) - \phi\left( m \right)_x   \right|  \leq \frac{\beta( m )_x}{n}+ \frac{1}{n} + \frac{2 K}{n}. 
\end{equation}

We are now ready to prove \eqref{eq_concentration_bound}. Having fixed
 $a \in \mathbb{R}_+$ and $M \in \mathbb{N}_0^{V}$, let 
\begin{equation}\label{K_def_eq}
    K:=  \frac{a n - \Vert \nu \Vert_1 - \Vert M \Vert_1/n -1 }{3}, \qquad U:=\left\{ u \in \mathbb{N}_0 \, : \, \left|u- \beta( m )_x \cdot n \right| \leq  K \right\}.
\end{equation} 
Recalling \eqref{EI_x_nbetax} we see that if $\left|I_x- \beta( m )_x \cdot n\right| \geq K$ implies $\left|I_x- \mathbb{E}(I_x) \right| \geq K-1$, thus 
\begin{multline}\label{restrict_to_well_behaved_u}
   \p\left( \left|\phi^{(n)}\left( m \right)_x - \phi\left( m \right)_x \right| \geq a \right) \stackrel{\eqref{dev1}}{\leq} \\
  \p\left( \left|\phi^{(n)}\left( m \right)_x - \phi\left( m\right)_x \right| \geq a, \; I_x \in U \right) +  2\exp\left(-\frac{2(K-1)_+^2}{  \Vert M\Vert_1}\right).
\end{multline}
Recalling that by \eqref{def_eq_scaled_phi} we have $\phi^{(n)}\left( m \right)_x=\frac{1}{n}\Phi^{(n)}_x$, we bound
\begin{multline}\label{split_bound_u}
   \p\left( \left|\phi^{(n)}\left( m \right)_x - \phi\left( m \right)_x \right| \geq a, \; I_x \in U \right) = \\
   \sum_{u \in U} \mathbb{P}_u\left( \left| \frac{1}{n}\Phi^{(n)}_x - \phi\left( m \right)_x \right| \geq a \right) \mathbb{P}(I_x=u) \stackrel{(*)}{\leq} \\
  \sum_{u \in U} \mathbb{P}_u\left( \left| \frac{1}{n}\Phi^{(n)}_x - \frac{1}{n}\mathbb{E}_u(\Phi^{(n)}_x) \right| \geq \frac{K}{n} \right)  \mathbb{P}(I_x=u) \stackrel{\eqref{dev2}}{\leq} \\ \sum_{u \in U}
 2 \exp\left(-\frac{2}{9}\frac{K_+^2}{n+u}\right) \mathbb{P}(I_x=u) \stackrel{(**)}{\leq} 
2  \exp\left(-\frac{2}{9}\frac{K_+^2}{n + \Vert \nu \Vert_1 n+ \Vert M \Vert_1 +K }\right),
 \end{multline}
where in  $(*)$ we used that $\beta( m )_x \leq \Vert \nu \Vert_1 + \frac{1}{n} \Vert M \Vert_1  $, thus for any $u \in U$ we have
\begin{equation}
   a- \left| \frac{1}{n}\mathbb{E}_u(\Phi^{(n)}_x) - \phi\left( m \right)_x   \right| \stackrel{\eqref{E_u_Phi_n_approx}}{\geq} a - \frac{\Vert \nu \Vert_1}{n} - \frac{\Vert M \Vert_1}{n^2} -\frac{1}{n} - \frac{2K}{n} \stackrel{\eqref{K_def_eq}}{=} \frac{3K}{n}-\frac{2K}{n} = \frac{K}{n}, 
\end{equation}
and in $(**)$ we used $\sum_{u \in U} \mathbb{P}(I_x=u) \leq 1$ and  that $u \in U$ implies  $ u \leq \Vert \nu \Vert_1 \cdot n+ \Vert M \Vert_1 +K $.

Putting the bounds \eqref{restrict_to_well_behaved_u} and \eqref{split_bound_u} together, also using the definition \eqref{K_def_eq} of $K$ and the union bound over $x \in V$ we obtain \eqref{eq_concentration_bound}.

{\bf Acknowledgements.} 
The authors would like to thank the Erd\H{o}s Center for sponsoring the Activated Random Walk 
focused workshop in May, 2025 where this research was initiated. 
J.R. was partially supported by a Simons Foundation Targeted Grant awarded
to the HUN-REN Alfréd Rényi Institute of Mathematics when this research was conducted. 
 M.S.\ was partially supported by the grant NKFI-FK-142124 of NKFI
(National Research, Development and Innovation Office). B.R.\ was partially supported by the grants ERC Synergy Grant No. 810115-DYNASNET and ADVANCED-150474 of NKFI. We thank Antal A.\ J\'arai and D\'aniel Keliger for valuable discussions.

\end{document}